\documentclass[a4paper,12pt]{article}

\usepackage{amsmath}
\usepackage{amsthm}
\usepackage{amsfonts}
\usepackage{amssymb}

\parskip\smallskipamount
\numberwithin{equation}{section} 

\newtheorem{theorem}{Theorem}

\newtheorem{thm}{Theorem}[section]
\newtheorem{lem}[thm]{Lemma}

\newtheorem{prop}[thm]{Proposition}

\theoremstyle{definition}
\newtheorem{definition}[theorem]{Definition}%chapter

\theoremstyle{remark}

\newcommand{\pr}{\mathbf{P}}
\newcommand{\E}{\mathbf{E}}
\newcommand{\R}{\mathbb{R}}

\newcommand{\Z}{\mathbb{Z}}

\title{On the growth of one-dimensional\\ reverse immunization contact processes} 
\date{}
\author{Achilleas Tzioufas\footnote{{\sc Heriot-Watt University.}}}

\begin{document} 

\maketitle

\begin{abstract}
We are concerned with the variation of the supercritical nearest neighbours contact process such that first infection occurs at a lower rate; it is known that the process survives with positive probability. Regarding the rightmost infected of the process started from one site infected and conditioned to survive, we specify a sequence of space-time points at which its behaviour regenerates and thus obtain the corresponding strong law and central limit theorem. We also extend complete convergence in this case. 
\end{abstract}

%\keywords{Contact processes; Kuczek's type argument} 

%\ams{60K35}{82C22}

\section{Introduction and main results}\label{S0}
We begin by defining a class of processes that includes the processes we are especially interested in. The  \textit{nearest neighbours three state contact process} with parameters $(\lambda,\mu)$ is a continuous time Markov process $\zeta_{t}$ with state space $\{-1,0,1\}^{\Z}$, elements of which are called configurations. 
The evolution of $\zeta_{t}$ is described locally as follows. Think of configurations as functions from $\Z$ to $\{-1,0,1\}$, transitions at each site $x$, $\zeta_{t}(x)$, occur according to the rules:
\begin{equation*}\label{rates}
\begin{array}{cl}
-1 \rightarrow 1 & \mbox{ at rate } \lambda \hspace{0.5mm} |\{ y = x-1,x+1 :\zeta_{t}(y)= 1\}|, \\
\mbox{ } 0 \rightarrow 1 & \mbox{ at rate } \mu \hspace{0.5mm} |\{ y = x-1,x+1 :\zeta_{t}(y)= 1\}|, \\   
\mbox{ } 1 \rightarrow 0 & \mbox{ at rate } 1,
\end{array}
\end{equation*}
for all times $t \geq0$, where $|B|$ denotes the cardinal of $B\subset \Z$.   Typically, the process started from configuration $\eta$ is denoted as $\zeta_{t}^{\eta}$. For general information about interacting particle systems, such as the fact that the above rates specify a well-defined process, we refer to Liggett \cite{L}. We note that the cases $\lambda = \mu$ and $\mu=0$ correspond to the extensively studied processes known as the contact process and as the forest fire model respectively, see $\mbox{e.g.}$ $\cite{L99}$, $\cite{D88}$. Furthermore in the literature various survival aspects of the three state contact process on the $d$-dimensional lattice were studied by Durrett and Schinazi \cite{DS} and by Stacey \cite{S}, the latter also includes results for the process on homogeneous trees. 

The process is thought of according to the following epidemiological interpretation. Given a configuration $\zeta$, each site  $x$  is regarded as infected if $\zeta(x) = 1$, as susceptible and never infected if $\zeta(x) = -1$ and, as susceptible and previously infected if $\zeta(x) = 0$. The \textit{standard initial configuration} is such that the origin is infected while all other sites are susceptible and never infected. We will use  $\zeta_{t}^{O}$ to denote the nearest neighbours three state contact process process started from the standard initial configuration. 
We say that the three state contact process \textit{survives} if $\pr(\zeta_{t}^{O} \mbox{ survives})>0$, where the event $\{\forall \hspace{0.5mm}t\geq0, \exists  \hspace{0.5mm} x: \zeta_{t}(x) =1\}$ is abbreviated as $\{\zeta_{t} \mbox{ survives}\}$. 

For $\zeta_{t}^{O}$ we have that transitions $-1 \rightarrow 1$,  $ 0 \rightarrow 1$, and $1 \rightarrow 0$ correspond respectively to initial infections,  subsequent infections and recoveries.  Accordingly, the initial infection of a site
induces a permanent alternation of the parameter proportional to which it will be susceptible; hence the parameter either decreases, corresponding to (partial) immunization, or increases, i.e.\ the reverse occurs. Our results concern the three state contact process under the constraint that $\mu\geq\lambda$, this explains the title of the article.
When modelling an epidemic,  the case that $\mu \leq \lambda$ could be a consequence of imperfect inoculation of individuals following their first exposure to the disease, while the case that $\mu \geq \lambda$ could be a consequence of debilitation of individuals caused by their first exposure to the disease. Specifically, tuberculosis and bronchitis are plausible examples of a disease that captures the latter characteristic.

When $(\lambda,\mu)$ are such that $\lambda = \mu$ the process is reduced to the well known contact process. In this case we will identify a configuration with the subset of $\Z$ that corresponds to the set of its infected sites, since states $-1$ and $0$ are effectively equivalent. Also, it is well known that the contact process exhibits a phase transition phenomenon, $\mu_{c}$ will denote its (one-dimensional nearest neighbours) critical value, i.e., $0<\mu_{c}<\infty$ and, if $\mu< \mu_{c}$ the process dies out while if $\mu> \mu_{c}$ the process survives, 
for an account of various related results and proofs see \cite{L}, \cite{D88} and \cite{L99}.

It is known that the three state contact process with parameters $(\lambda,\mu)$ such that $\mu > \mu_{c}$ and $\lambda>0$ survives, see \cite{DS}. We are concerned with the behaviour of the process when survival occurs 
assuming additionally that $\mu\geq\lambda$. The following theorem summarizes the main results of this paper, 
in words, parts \textit{(i)} and \textit{(ii)}  are respectively a law of large numbers and the corresponding central limit theorem for the rightmost infected while parts \textit{(iii)} and \textit{(iv)} are respectively a law of large numbers and complete convergence for the set of infected sites of the process. For demonstrating our results we introduce some notation.  The standard normal distribution function is represented by $N(0,\sigma^{2}), \sigma^{2}>0,$ also, weak convergence of random variables and of set valued processes are denoted by "$\overset{w}{\rightarrow}$" and by "$\Rightarrow$" respectively. Further, we denote by $\bar{\nu}_{\mu}$ the upper invariant measure of the contact process with parameter $\mu$,  and by $\delta_{\emptyset}$ the probability measure that puts all mass on the empty set. (For general information about the upper invariant measure and weak convergence of set valued processes we refer to Liggett \cite{L99}).

\begin{thm}\label{THEthm} 
Consider $\zeta_{t}^{O}$ with parameters $(\lambda, \mu)$, and let $I_{t}= \{x: \zeta_{t}^{O}(x)=1\}$ and $r_{t} = \sup I_{t}$. If $(\lambda, \mu)$ are such that $\mu \geq \lambda>0$ and $\mu > \mu_{c}$ then there exists $\alpha>0$ such that conditional on $\{\zeta_{t}^{O} \mbox{\textup{ survives}}\}$,

(i) $\displaystyle{  \frac{r_{t}}{t} \rightarrow \alpha}$, almost surely;

(ii) $\displaystyle{ \frac{r_{t} - \alpha t}{\sqrt{t}} \overset{w}{\rightarrow} N(0,\sigma^{2})}$, for some $\sigma^{2}>0$;

(iii) let $\theta= \theta(\mu)$ be the density of $\bar{\nu}_{\mu}$, then, $\displaystyle{ \frac{|I_{t}|}{t} \rightarrow 2 \alpha \theta}$, almost surely. 

(iv) Let $\beta= \pr(\zeta_{t}^{O} \mbox{ survives})>0$, then, $\displaystyle{ I_{t} \Rightarrow (1-\beta) \delta_{\emptyset} + \beta \bar{\nu}_{\mu}}$.
\end{thm}

We comment on the proof of Theorem \ref{THEthm}. The cornerstone for acquiring parts  \textit{(i)\textup{ and} (ii)} is to ascertain the existence of a sequence of space-time points, termed \textit{break points}, strictly increasing in both space and time, among which the behaviour of $r_{t}$ conditional on $\{\zeta_{t}^{O} \mbox{\textup{ survives}}\}$ stochastically replicates; these type of arguments have been established by Kuczek, see \cite{K}. We also note that proofs of parts \textit{(iii) \textup{and} (iv)} are based on variations of the arguments for the contact process case due to Durrett and Griffeath, see \cite{DG} and \cite{D80}, \cite{G}.

In the next section we introduce the graphical construction, we also present monotonicity and give some elementary coupling results. Section \ref{Sexp} is intended for the proof of two exponential estimates that we need for latter. Section \ref{S3} is devoted to the study of break points and the proof of Theorem \ref{THEthm}.

\section{Preliminaries}\label{S2}

\subsection{The graphical construction}\label{grrep}
The graphical construction will be used in order to visualize the construction of various processes on the same probability space; we will repeatedly use it throughout this paper.

Consider parameters $(\lambda,\mu)$ and suppose that $\mu \geq \lambda$, the other case is similar.  To carry out our construction for all sites $x$ and $y=x-1,x+1$, let $(T_{n}^{x,y})_{n\geq1}$ and $(U_{n}^{x,y})_{n\geq1}$ be the event times of Poisson processes respectively at rates $\lambda$ and $\mu-\lambda$; further, let $(S_{n}^{x})_{n\geq 1}$ be the event times of a Poisson process at rate $1$. (All Poisson processes introduced are independent). 

Consider the space $\Z \times [0,\infty)$ thought of as giving a time line to each site of $\Z$; Cartesian product is denoted by $\times$. Given a realization of the before-mentioned ensemble of Poisson processes, we define the \textit{graphical construction} and $\zeta_{t}^{[\eta,s]}$,  $t\geq s$, the nearest neighbours three state contact process with parameters $(\lambda,\mu)$ started from $\eta$ at time $s\geq0$, i.e.\ $\zeta_{s}^{[\eta,s]} = \eta$, as follows. At each point $x \times T_{n}^{x,y}$ we place a directed $\lambda$-\textit{arrow} to $y \times T_{n}^{x,y}$; this indicates that at all times $t=T_{n}^{x,y}$, $t\geq s$, if $\zeta_{t-}^{[\eta,s]}(x)=1$ and $\zeta_{t-}^{[\eta,s]}(y)=0$ \textit{or} $\zeta_{t-}^{[\eta,s]}(y) = -1$ then we set $\zeta_{t}^{[\eta,s]}(y) = 1$ (where $\zeta_{t-}(x)$ denotes the limit of $\zeta_{t-\epsilon}(x)$ as $\epsilon\rightarrow 0$). At each point $x \times U_{n}^{x,y}$ we place a directed $(\mu-\lambda)$-\textit{arrow} to $y \times U_{n}^{x,y}$; this indicates that at any time $t=U_{n}^{x,y}$, $t\geq s$, if $\zeta_{t-}^{[\eta,s]}(x)=1$ and $\zeta_{t-}^{[\eta,s]}(y)=0$ then we set $\zeta_{t}^{[\eta,s]}(y) = 1$. While at each point $x \times S_{n}^{x}$ we place a \textit{recovery mark}; this indicates that at any time $t= S_{n}^{x}, t\geq s,$ if $\zeta_{t-}^{[\eta,s]}(x)=1$ then we set $\zeta_{t}^{[\eta,s]}(x) = 0$.  The reason we introduced the special marks is to make connection with percolation and hence the contact process, we define  the contact process $\xi_{t}^{A}$  with parameter $\mu$ started from $A \subset \Z$ as follows. We write $A \times 0 \rightarrow B \times t$, $t\geq 0$, if there exists a connected oriented path from $x \times 0$ to $y \times t$, for some $x \in A$ and $y \in B$, that moves along arrows (of either type) in the direction of the arrow and along time lines in increasing time direction without passing through a recovery mark, defining $\xi_{t}^{A} := \{x: A \times 0 \rightarrow x \times t\}$, $t\geq0$, we have that $(\xi_{t}^{A})$ is a set valued version of the contact process with parameter $\mu$ started from $A$ infected.

It is important to emphasize that the graphical construction, for fixed $(\lambda,\mu)$, defines all  $\zeta_{t}^{[\eta,s]}$,  $t\geq s$, for any configuration $\eta$ and time $s\geq0$, and all $\xi_{t}^{A}$, for any $A \subset \Z$, simultaneously on the same probability space, i.e.\ provides a coupling of all these processes.

\definition{We shall denote by $\mathcal{I}(\zeta)$ the set of infected sites of any given configuration $\zeta$, i.e. $\mathcal{I}(\zeta) = \{y \in \Z:\zeta(y) =1\}$.}\label{calI}

To simplify our notation, consistently to Section \ref{S0}, $\zeta_{t}^{[\eta,0]}$ is denoted as $\zeta_{t}^{\eta}$, and, letting $\eta_{0}$ be the standard initial configuration, $\zeta_{t}^{[\eta_{0},0]}$ is denoted as $\zeta_{t}^{O}$. Additionally, the event $\{\mathcal{I}(\zeta_{t}^{[\eta,s]}) \not= \emptyset  \mbox{ for all } t\geq s\}$ will be abbreviated below as  $\{\zeta_{t}^{[\eta,s]} \mbox{ survives}\}$. 

Finally, we note that we have produced a version of $\zeta_{t}^{\eta}$ via a countable collection of Poisson processes, this provides well-definedness of the process. Indeed, whenever one assumes that $|\mathcal{I}(\eta)|<\infty$, this is a consequence of standard Markov chains results having an almost surely countable state space; otherwise, this is provided by an argument due to Harris \cite{H}, see Theorem 2.1 in Durrett \cite{D95}.

\subsection{Monotonicity, coupling results}\label{coupling}

To introduce monotonicity concepts, we endow the space of configurations $\{-1,0,1\}^{\Z}$ with the \textit{component-wise partial order}, $\mbox{i.e.}$, for any two configurations $\eta_{1}, \eta_{2}$ we have that $\eta_{1} \leq \eta_{2}$ whenever $\eta_{1}(x) \leq \eta_{2}(x)$ for all $x \in \Z$. The following theorem is a known result, for a proof we refer to section 5 in Stacey \cite{S}.

\begin{thm}\label{moninit1}
Let $\eta$ and $\eta'$ be any two configurations such that $\eta\leq \eta'$. Consider the respective three state contact processes $\zeta_{t}^{\eta}$ and $\zeta_{t}^{\eta'}$ with the same parameters $(\lambda,\mu)$ coupled by the graphical construction. For all $(\lambda,\mu)$ such that $\mu \geq \lambda>0$, we have that $\zeta_{t}^{\eta} \leq \zeta_{t}^{\eta'}$ holds. We refer to this property as monotonicity in the initial configuration.
\end{thm}

For the remainder of this subsection we give various coupling results concenring $\zeta_{t}^{O}$ the nearest neighbours three state contact process with parameters $(\lambda,\mu)$ started from the standard initial configuration, let $I_{t} = \mathcal{I}(\zeta_{t}^{O})$, $r_{t}= \sup I_{t}$ and $l_{t}= \inf I_{t}$.

We note that both the nearest neighbours assumption as well as the assumption that $\mu \geq \lambda$ in all three of the proofs in this subsection is crucial.

The next lemma will be used repeatedly throughout this paper, its proof given below is a simple extension of a well known result for the nearest neighbours contact process on $\Z$, see  $\mbox{e.g.}$ $\cite{D80}$. 

\begin{lem}\label{piprendcoup1}
Let $\eta'$ be any configuration such that $\eta'(0)= 1$ and $\eta'(x)= -1$ for all $x\geq1$. Consider   $\zeta_{t}^{\eta'}$ with parameters $(\lambda,\mu)$ and let $r_{t}'= \sup\mathcal{I}(\zeta_{t}^{\eta'})$. For  $(\lambda,\mu)$ such that $\mu \geq \lambda$, if $\zeta_{t}^{O}$ and $\zeta_{t}^{\eta'}$ are coupled by the graphical construction then the following property holds, for all $t\geq 0$,
\begin{equation*}\label{couprend2}
r_{t} = r_{t}'  \mbox{ on } \{ I_{t} \not= \emptyset\}.
\end{equation*} 
\end{lem}

\begin{proof}
We prove the following stronger statement, for all $t\geq 0$, 
\begin{equation}\label{couprend1}
\zeta_{t}^{O}(x) = \zeta_{t}^{\eta'}(x)   	\mbox{ for all } x \geq l_{t}, \mbox{ on }  \{I_{t} \not= \emptyset\}.
\end{equation} 
Clearly (\ref{couprend1}) holds for $t=0$, we show that all possible transitions preserve it.  An increase of $l_{t}$ (i.e., a recovery mark at $l_{t}\times t$) as well as any transition changing the state of any site $x$ such that $ x\geq l_{t}+1$  preserve $(\ref{couprend1})$.  It remains to examine transitions that decrease $l_{t}$, by monotonicity in the initial configuration we have that the possible pairs of $(\zeta_{t}^{O}(l_{t}-1), \zeta_{t}^{\eta'}(l_{t}-1))$ are the following $(-1,0),(-1,1),(0,0),(0,1)$. In the first pair case $(\ref{couprend1})$ is preserved because $\lambda$-arrows are used for transitions $-1 \rightarrow 1$ as well as $ 0 \rightarrow 1$, while in the three remaining cases this is obvious, the proof of $(\ref{couprend1})$ is thus complete. 
\end{proof}

The next lemma will be used in the proof of the two final parts of Theorem $\ref{THEthm}$, its proof is a simple variant of that of Lemma $\ref{piprendcoup1}$ and thus is omitted.

\begin{lem}\label{cccoup}
Let $\xi_{t}^{\Z}$ be the nearest neighbours contact process with parameter $\mu$ started from $\Z$. For $(\lambda,\mu)$ such that $\mu \geq \lambda>0$, if $\zeta_{t}^{O}$ and $\xi_{t}^{\Z}$ are coupled by the graphical construction the following property holds, for all $t\geq 0$,
\begin{equation*}
I_{t} =  \xi_{t}^{\Z} \cap [l_{t},r_{t}] \mbox{ \textup{on} } \{I_{t} \not= \emptyset\}.
\end{equation*}
\end{lem}

\definition{ For all integer $k$, let $\eta_{k}$ be the configuration such that $\eta_{k}(k) = 1$ and $\eta_{k}(y) =-1$ for all $y \not= k$.}\label{defetak}

Our final coupling result will be used in the definition of break points in Subsection $\ref{S31}$. 
To state the lemma, define the stopping times $\tau_{k} = \inf\{t: r_{t}=k\}$, $k\geq1$, and also $R = \sup_{t\geq0}r_{t}$.

\begin{lem}\label{Sk}
Let $(\lambda,\mu)$ be such that $\mu \geq \lambda>0$ and consider the graphical construction. Consider also the processes $\zeta_{t}^{[\eta_{k},\tau_{k}]}$, $k\geq1$, started at times $\tau_{k}$ from $\eta_{k}$, as in Definition \ref{defetak}. Then, for all $\mbox{ }k=1,\dots,R$ the following property holds, 
\begin{equation*}\label{eqSk}
\zeta_{t}^{O} \geq \zeta_{t}^{[\eta_{k}, \tau_{k}]}, \mbox{ for all } t \geq \tau_{k}.
\end{equation*}
\end{lem}

\begin{proof}
We have that $\zeta_{\tau_{k}}^{O}(k) =1$, because $\eta_{k}$ is the least infectious configuration such that $\eta_{k}(k)=1$, we also have $\zeta^{O}_{\tau_{k}} \geq \eta_{k}$ $\mbox{ for all }k=1,\dots,R,$ by monotonicity in the initial configuration the proof is complete. 
\end{proof}

\section{Exponential estimates}\label{Sexp}

This section is intended for proving two exponential estimates for three state contact processes that will be needed in Section $\ref{S3}$. The method used is based on a renormalization result of Durrett and Schinazi $\cite{DS}$ that is an extension of the well-known work of Bezuidenhout and Grimmett \cite{BG}.

Subsequent developments require understanding of oriented site percolation.
Consider the set of sites, $\mathbb{L} =\{ (y,n) \in \Z^{2} :  n \geq0 \mbox{ and } y+n \mbox{ is even}\}.$
For each site $(y,n) \in \mathbb{L}$ we associate an independent Bernoulli random variable $w(y,n)\in \{0,1\}$ with parameter $p>0$; if $w(y,n)=1$ we say that $(y,n)$ is \textit{open}. We write $(x,m)\rightarrow(y,n)$ whenever there exists a sequence of open sites $(x,m) \equiv (y_{0}, m),\dots, (y_{n-m}, n)  \equiv (y,n)$ such that and $|y_{i} - y_{i-1}|=1$ for all $i=1,\dots, n-m$. Define $(A_{n})_{n\geq0}$ with parameter $p$ as $A_{n} = \{y:(0,0) \rightarrow (y,n)\}$. We write $\{A_{n} \textup{ survives}\}$ as an abbreviation for $\{\forall \hspace{0.5mm} n\geq1: A_{n} \not= \emptyset \}$. 

The next proposition is the renormalization result, it is a consequence of Theorem  4.3 in Durrett \cite{D95}, where the comparison assumptions there hold due to Proposition 4.8 of Durrett and Schinazi $\cite{DS}$. For stating it, given constants $L,T$, we define the set of configurations $Z_{y} = \{ \zeta :  |\mathcal{I}(\zeta) \cap [-L+2Ly,L+2Ly] | \geq L^{0.6} \}$, for all integers $y$.

\begin{prop}\label{couplDS}
Let $\eta$ be any configuration such that $\eta \in Z_{0}$, consider $\zeta_{t}^{\eta}$ with parameters $(\lambda, \mu)$ such that $\mu >\mu_{c}$ and $\lambda>0$. For all $p<1$ there exist constants $L,T$ such that $\zeta_{t}^{\eta}$ can be coupled to $A_{n}$ with parameter $p$ so that, 
\[
y \in A_{n} \hspace{1mm} \Rightarrow \hspace{1mm} \zeta_{nT}^{\eta} \in Z_{y}
\]
$(y,n) \in \mathbb{L}$. In particular the process survives. 
\end{prop}

The first of the exponential estimates that we need for Section $\ref{S3}$ is the following.

\begin{prop}\label{expbounds3scp}
Consider $\zeta_{t}^{O}$ with parameters $(\lambda, \mu)$. Let also $I_{t} = \mathcal{I}(\zeta_{t}^{O})$, $r_{t} = \sup I_{t}$ and $R= \sup_{t\geq0} r_{t}$, further let $\rho = \inf\{t: I_{t} = \emptyset\}$. If $(\lambda, \mu)$ are such that $\mu >\mu_{c}$ and $\mu\geq\lambda>0$ then there exist constants $C$ and $\gamma>0$ such that
\begin{equation}\label{eqexpbounds3scp}
\pr(R  \geq n, \mbox{ }\rho<\infty) \leq Ce^{-\gamma n}, 
\end{equation}
for all $n\geq1$.
\end{prop}

\begin{proof}
Consider the graphical construction for  $(\lambda, \mu)$ as in the statement. Recall the component-wise partial order on the space of configurations, the property of monotonicity in the initial configuration that were introduced in subsection $\ref{coupling}$ and,  the configurations $\eta_{k}$ as in Definition \ref{defetak}. By Proposition $\ref{couplDS}$, emulating the proof of Theorem 2.30 (a) of Liggett \cite{L99}, we have that 
\begin{equation}\label{rhoconfin}
\mbox{ } \pr( t < \rho < \infty) \leq Ce^{-\gamma t}, 
\end{equation}
for all $t\geq0$; to see that the arguments given there apply in this context note that, by monotonicity in the initial configuration, for any time $s>0$ and any $x\in I_{s}$, considering the process $\zeta_{t}^{ [\eta_{x}, s]}$ we have that $\zeta_{t}^{O} \geq  \zeta_{t}^{[\eta_{x}, s]}$ for all $t\geq s$, hence, the proof we referred to applies for $\delta = \pr(\zeta_{1}^{O} \in Z_{0}) >0$.  

For proving $(\ref{eqexpbounds3scp})$, by set theory we have that for all $n \geq 1$,
\begin{equation*}\label{RT}
\pr\left(R  > n, \rho < \infty\right) \leq \pr\left( \frac{n}{\lambda} < \rho< \infty\right) + \pr\left(\rho < \frac{n}{\lambda}, \mbox{ } R  > n\right)
\end{equation*}
the first term on the right hand side decays exponentially in $n$ due to $(\ref{rhoconfin})$, thus, it remains to prove that the probability of the event $\{\sup_{t\leq \frac{n}{\lambda}} r_{t} >n\}$ decays exponentially in $n$, which however is immediate because $\sup_{t \in (0,u]} r_{t}$ is bounded above in distribution by the number of events by time $u$ in a Poisson process at rate $\lambda$ and standard large deviations results for the latter. 
\end{proof}

The following elementary result for independent site percolation as well as the subsequent geometrical lemma are needed in the proof of Proposition $\ref{shadldev2}$ below.

\begin{lem}\label{indperc}
Consider $(A_{n})$ with parameter $p$ and define  $R_{n} = \sup A_{n}, n\geq0$.
For $p$ sufficiently close to 1 there are strictly positive and finite constants $a,\gamma$ and $C$ such that
\begin{equation*}\label{Rnspproc}
\pr( R_{n} < an, \mbox{ } A_{n} \textup{ survives}) \leq Ce^{- \gamma n}, 
\end{equation*}
for all $n\geq1$. 
\end{lem}

\begin{proof}
Define $A'_{n} = \{y: (x,0) \rightarrow (y,n) \mbox{ for some } x \leq0\}$ and let $R'_{n}= \sup A'_{n}$, $n\geq1$.
Because $R_{n} = R'_{n}$ on $\{A_{n} \textup{ survives}\}$, it is sufficient to prove that $p$ can be chosen sufficiently close to 1 such that, for some $a>0$, the probability of the event $R'_{n} < an$ decays exponentially in $n \geq 0$. Letting $B_{n}'$ be independent oriented bond percolation on $\mathbb{L}$ with supercritical parameter $\tilde{p}<1$ started from $\{(x,0) \in \mathbb{L} : x\leq0 \}$, the result follows from the corresponding large deviations result for $B_{n}'$ (see Durrett \cite{D84}, (1) in section 11), because for $p= \tilde{p}(2-\tilde{p})$ we have that $B_{n}'$ can be coupled to $A_{n}'$ such that $B_{n}' \subset A_{n}'$ holds, see Liggett \cite{L99}, p.13.
\end{proof}

\begin{lem}\label{geomRR+}
Let $b,c$ be strictly positive constants. For any $a < c$ we can choose sufficiently small $\phi \in (0,1)$, that does not depend on $t \in \R\geq0$, such that for all $x \in [-b \phi t,b \phi t]$, 
\begin{equation}\label{eq:geom}
[x - c (1-\phi)t, x + c (1-\phi)t] \supseteq [-at,at],
\end{equation}
$t\geq0$.
\end{lem} 

\begin{proof}
Note that it is sufficient to consider $x = brt$; then, simply choose $\phi$ such that $btr -c(1-\phi)t< - at$, $\mbox{i.e.}$ for 
$\displaystyle{ \phi < \frac{c-a}{c+b}}$, $\phi>0$, equation $(\ref{eq:geom})$ holds. 
\end{proof}

The other exponential estimate we will need in Section $\ref{S3}$ is the following.

\begin{prop}\label{shadldev2}
Let  $\bar{\eta}$ such that $\bar{\eta}(x) =1$ for all $x \leq 0$ while $\bar{\eta}(x) =-1$ otherwise. Consider $\zeta^{\bar{\eta}}_{t}$ with parameters $(\lambda, \mu)$ and let $\displaystyle{ \bar{r}_{t} =\sup \mathcal{I}(\zeta^{\bar{\eta}}_{t})}$. If $(\lambda, \mu)$ are such that $\mu >\mu_{c}$ and $\mu\geq\lambda>0$ then there exist strictly positive and finite constants $a,\gamma$ and $C$ such that
\[
\pr\left(\bar{r}_{t}  < at\right) \leq Ce^{-\gamma t},
\]
for all $t\geq0$.
\end{prop}

\begin{proof} 
Consider the graphical construction for  $(\lambda, \mu)$ as in the statement. Let $p$ be sufficiently close to $1$ so that Lemma $\ref{indperc}$ is satisfied. Recall the configurations $\eta_{x}$ as in Definition \ref{defetak}. By the proof of Theorem 2.30 (a) of Liggett \cite{L99}---which applies for the reasons explained in the first paragraph of the proof of Proposition $\ref{expbounds3scp}$---, we have that total time $\sigma$ until we get a percolation process $A_{n}$ with parameter $p$ that is coupled to $\zeta_{t}^{[\eta_{\bar{r}_{\sigma}},\sigma]}$ as explained in Proposition $\ref{couplDS}$ (for $\bar{r}_{\sigma} \times (\sigma +1)$ being thought of as the origin) and is conditioned on $\{A_{n} \textup{ survives}\}$, is exponentially bounded. From this, because $\bar{r}_{t}$ is bounded above in distribution by a Poisson process,  we have that there exists a constant $\tilde\lambda$ such that the event 
$\left\{\bar{r}_{\sigma} \times (\sigma +1) \in  [ -\tilde{\lambda}t d, \tilde{\lambda}td] \times (0, td]\right\}$, for all $d \in (0,1)$, occurs outside some exponentially small probability in $t$. Finally on this event, by Lemma $\ref{indperc}$ and the coupling in Lemma $\ref{piprendcoup1}$, we have that there exists an $\tilde{a}>0$ such that $\bar{r}_{t} \geq \tilde{a}t - \bar{r}_{\sigma}$, again outside some exponentially small probability in $t$, choosing $\tilde{\lambda}=b$  and $\tilde{a}=c$ in Lemma $\ref{geomRR+}$ completes the proof.%---in fact the result then holds for all $a<\tilde{a}$. 

\end{proof}

\section{Main Results}\label{S3} 
This section is organized as follows. In Subsection \ref{S31} we prove Theorem \ref{THEprop} stated below; based on this theorem, we prove Theorem \ref{THEthm} in Subsection \ref{S32}.

\begin{thm}\label{THEprop}
Consider $\zeta_{t}^{O}$ with parameters $(\lambda, \mu)$ and let $r_{t} = \sup\mathcal{I}(\zeta_{t}^{O})$. Suppose $(\lambda, \mu)$ such that $\mu >\mu_{c}$ and $\mu\geq\lambda>0$. On $\{\zeta^{O}_{t} \mbox{\textup{ survives}}\}$ there exist random (but not stopping) times $\tilde{\tau}_{0}:=0 < \tilde{\tau}_{1} < \tilde{\tau}_{2} < \dots$ such that $(r_{\tilde{\tau}_{n}} -r_{\tilde{\tau}_{n-1}}, \tilde{\tau}_{n}- \tilde{\tau}_{n-1})_{n \geq 1}$ are i.i.d.\ random vectors, where also $r_{\tilde{\tau}_{1}} \geq 1$ and $\displaystyle{r_{\tilde{\tau}_{n}} = \sup_{t\leq \tilde{\tau}_{n}} r_{t}}$. Furthermore, letting $M_{n}=r_{\tilde{\tau}_{n}}-\inf_{t \in[ \tilde{\tau}_{n}, \tilde{\tau}_{n+1})} r_{t}$, $n \geq 0$, we  have that $(M_{n})_{n\geq0}$ are i.i.d.\ random variables, where also $M_{n} \geq0$. Finally,  $r_{\tilde{\tau_{1}}}, \tilde{\tau}_{1},M_{0}$ are exponentially bounded. 
\end{thm}

\subsection{Break points}\label{S31} 

We first define our break points.

\definition{Consider the graphical construction for $(\lambda, \mu)$ such that $\mu >\mu_{c}$ and $\mu\geq\lambda>0$. 
Consider $\zeta_{t}^{O}$, define $r_{t} =\sup \mathcal{I}(\zeta_{t}^{O})$, define also the stopping times $\tau_{k} = \inf\{t: r_{t}=k\}$, $k\geq0$. Let $\eta_{k}$ be as in Definition \ref{defetak}. Our break points, which we are about to define, is the unique strictly increasing, in space and in time, subsequence of the space-time points $k \times \tau_{k}, k\geq1,$ such that  $\zeta_{t}^{[\eta_{k},\tau_{k}]} \textup{ survives}$. The origin $0 \times 0$ is a break point, i.e.\ our subsequence is identified on $\{\zeta_{t}^{O} \textup{ survives}\}$. Define $(K_{0}, \tau_{K_{0}})= (0, 0)$. For all $n\geq0$ and $K_{n}<\infty$ we inductively define
\[
K_{n+1} = \inf\{k \geq K_{n}+1:  \zeta_{t}^{[\eta_{k},\tau_{k}]} \mbox{ survives}\}, 
\]  
and $X_{n+1} = K_{n+1}-K_{n}$, additionally we define $\Psi_{n+1} = \tau_{K_{n+1}} - \tau_{K_{n}}$, and also
$\displaystyle{M_{n} =K_{n} -\inf_{\tau_{K_{n}} \leq t < \tau_{K_{n+1}}} r_{t}}$. The space-time points $K_{n} \times \tau_{K_{n}}$, $n\geq0$, are our  \textit{break points}.}\label{definbpts}

Letting $\tilde{\tau}_{n}:= \tau_{K_{n}}, n\geq0$, in the definition of break points above gives us that for proving Theorem $\ref{THEprop}$ it is sufficient to prove the two propositions following; this subsection is intended for proving these.

\begin{prop}\label{PROPexpbnd}
$K_{1}$, $\tau_{K_{1}}$ and $M_{0}$ are exponentially.
\end{prop}

\begin{prop}\label{PROPiid}
$(X_{n}, \Psi_{n} ,M_{n-1})_{n \geq 1}$, are independent identically distributed vectors.
\end{prop} 

\definition{Given a configuration $\zeta$ and an integer $y\geq1$, define the configuration $\zeta-y$ to be $(\zeta-y)(x) = \zeta(y+x)$, for all $x \in \Z$.}

We shall denote by $\mathcal{F}_{t}$ the sigma algebra associated to the ensemble of Poisson processes used for producing the graphical construction up to time $t$.

The setting of the following lemma is important to what follows.

\begin{lem}\label{bptsinf}
Let  $\bar{\eta}$ such that $\bar{\eta}(x) =1$ for all $x \leq 0$ while $\bar{\eta}(x) =-1$ otherwise. Consider $\zeta^{\bar{\eta}}_{t}$ with parameters $(\lambda, \mu)$. Define $\bar{r}_{t} =\sup\mathcal{I}(\zeta^{\bar{\eta}}_{t})$, define also, the stopping times $T_{n} = \inf\{t: \bar{r}_{t} = n\}$, $n\geq0$. Let  $(\lambda,\mu)$ be such that $\mu \geq \lambda>0$ and $\mu > \mu_{c}$ and consider the graphical construction. 

Let $Y_{1}:=1$ and consider $\zeta_{t}^{1}:=\zeta_{t}^{[\eta_{Y_{1}}, T_{1}]}$, we let $\rho_{1} = \inf\{ t\geq T_{1}: \mathcal{I}(\zeta_{t}^{1}) = \emptyset\}$. 
For all $ n\geq 1$, proceed inductively: On the event $\{\rho_{n} < \infty\}$ let
\[
Y_{n+1} = 1+ \sup_{t \in [T_{Y_{n}},\rho_{n})} \bar{r}_{t}, 
\]
and consider $\zeta_{t}^{n+1}:= \zeta_{t}^{[\eta_{Y_{n+1}},T_{Y_{n+1}}]}$, we let $\rho_{n+1} = \inf\{t \geq T_{Y_{n+1}}: \mathcal{I}(\zeta_{t}^{n+1}) = \emptyset\}$; on the event that $\{\rho_{n} = \infty\}$ let $\rho_{l} = \infty$ for all $l > n$. Define the random variable $N = \inf\{n\geq1: \rho_{n}= \infty\}$. We have the following expression,
\begin{equation}\label{Yinf}
Y_{N} = \inf\{k\geq1: \zeta_{t}^{[\eta_{k}, T_{k}]} \mbox{ \textup{survives}}\},
\end{equation}
and also,
\begin{equation}\label{eq:algopiprendcoup}
\bar{r}_{t} = \sup\mathcal{I}(\zeta_{t}^{n}), \mbox{ for all } t \in [T_{Y_{n}},\rho_{n}) \mbox{ and } n\geq1.
\end{equation}

We further have that
\begin{equation}\label{cbpts1}
(\zeta^{1}_{t+ T_{1}}- 1)_{t \geq 0} \mbox{ is independent of }\mathcal{F}_{T_{1}}   \mbox{ and is equal in distribution to } (\zeta_{t}^{O})_{t\geq0},
\end{equation}
and also, 
\begin{eqnarray}\label{cbpts2}
&& \mbox{ conditional on } \{\rho_{n}<\infty, Y_{n+1} = w\}, w \geq 1, (\zeta^{n+1}_{t+ T_{Y_{n+1}}}- w)_{t \geq 0} \nonumber\\
&& \mbox{ is independent of } \mathcal{F}_{T_{Y_{n+1}}} \mbox{ and is equal in distribution to } (\zeta_{t}^{O})_{t\geq0}.
\end{eqnarray} 
\end{lem}

\begin{proof}

Equation $(\ref{Yinf})$ is a consequence of Lemma $\ref{Sk}$, to see this note that this lemma gives that for all $n\geq1$ on $\{\rho_{n}<\infty\}$, $\rho_{n} \geq \inf\{t\geq T_{k}: \mathcal{I}(\zeta_{t}^{[\eta_{k},T_{k}]})= \emptyset\}$ $\mbox{for all } k=Y_{n}+1,\dots,Y_{n+1}-1$. Equation $(\ref{eq:algopiprendcoup})$ is immediate due to Lemma $\ref{piprendcoup1}$.

Note that from Proposition $\ref{shadldev2}$ we have that $T_{n} < \infty \mbox{ for all }  n\geq0$ a.s.. Then, equation $(\ref{cbpts1})$ follows from the strong Markov property at time $T_{1}<\infty$ and translation invariance; while $(\ref{cbpts2})$ is also immediate by applying the strong Markov property at time $T_{Y_{n+1}}<\infty$, where $T_{Y_{n+1}}<\infty$ because from Proposition $\ref{expbounds3scp}$ we have that, conditional on $\rho_{n}<\infty$,  $Y_{n+1} < \infty$ $\mbox{a.s.}$.
\end{proof}

The connection between the break points and Lemma $\ref{bptsinf}$ comes by the following coupling result.
 
\begin{lem}\label{KeqK'}
Let $\eta'$ be any configuration such that $\eta'(0)= 1$ and $\eta'(x)= -1$ for all $x\geq1$. Consider   $\zeta_{t}^{\eta'}$ with parameters $(\lambda,\mu)$ and let $r_{t}'= \sup\mathcal{I}(\zeta_{t}^{\eta'})$, let also $\tau'_{k} = \inf\{t\geq 0: r'_{t} = k\}$, $k\geq1$. Define the integers 
\[
K' = \inf\{k\geq 1: \zeta_{t}^{[\eta_{k}, \tau'_{k}]} \mbox{\textup{ survives}}\}, 
\]
and also $M'= \inf_{0 \leq t \leq \tau_{K}'} r_{t}'.$ Consider further $\zeta_{t}^{O}$  with parameters $(\lambda,\mu)$. For $(\lambda,\mu)$ such that $\mu \geq \lambda>0$ and $\mu >\mu_{c}$, if $\zeta_{t}^{O}$ and $\zeta_{t}^{\eta'}$ are coupled by the graphical construction the following property holds,
\begin{equation*}\label{coupK1tauK1}
 (K', \tau'_{K'},M') = (K_{1}, \tau_{K_{1}}, M_{0}), \mbox{ on  } \{\zeta_{t}^{O} \mbox{\textup{ survives}}\},
\end{equation*} 
where $K_{1}, \tau_{K_{1}}, M_{0}$ are as in Definition $\ref{definbpts}$. 
\end{lem}

The proof of Lemma \ref{KeqK'} is trivial, it is an immediate consequence of Lemma $\ref{piprendcoup1}$.

\begin{proof}[proof of Proposition $\ref{PROPexpbnd}$]
Consider the setting of Lemma $\ref{bptsinf}$. By the definition of break points, Definition \ref{definbpts}, and Lemma $\ref{KeqK'}$ we have that on $\{\zeta_{t}^{O} \mbox{ survives}\}$, $K_{1}=Y_{N}$, $\tau_{K_{1}}=  T_{Y_{N}}$ and $M_{0}=\inf_{t \leq T_{Y_{N}}}\bar{r}_{t}$. It is thus sufficient to prove that the random variables $Y_{N}, T_{Y_{N}}, \inf_{t \leq T_{Y_{N}}}\bar{r}_{t}$ are exponentially bounded, merely because an exponentially bounded random variable is again exponentially bounded conditional on any set of positive probability.

We have that 
\begin{equation}\label{eq:YN}
Y_{N} = 1+ \sum\limits_{n=2}^{N} (Y_{k} - Y_{k-1}) \mbox{ on } \{N\geq2\},
\end{equation} 
while $Y_{1}:=1$, using this and Proposition \ref{expbounds3scp}, we will prove that $Y_{N}$ is bounded above in distribution by a geometric sum of $\mbox{i.i.d.}$ exponentially bounded random variables and hence is itself exponentially bounded. 

Let $\rho$ and $R$ be as in Proposition \ref{expbounds3scp}, we define $p_{R}(w)= \pr(R +1 = w, \rho<\infty)$, and $\bar{p}_{R}(w)= \pr(R +1= w | \mbox{ }\rho <\infty)$, for all integer $w\geq1$, define also $p= \pr(\rho = \infty)>0$ and $q=1-p$, where $p>0$ by Proposition \ref{couplDS}.

By $(\ref{cbpts1})$ of the statement of Lemma $\ref{bptsinf}$, we have that
\begin{equation}\label{eq:Y2}
\pr( Y_{2} -Y_{1} = w, \rho_{1} < \infty) = p_{R}(w) 
\end{equation}
$w\geq1$; similarly, from $(\ref{cbpts2})$ of the same statement, we have that, for all $n\geq1$,  
\begin{equation}\label{eq:rec1}
\pr(\rho_{n+1} = \infty| \mbox{ } \rho_{n} <\infty, Y_{n+1} = z, \mathcal{F}_{T_{Y_{n+1}}}) = p, 
\end{equation}
and also, 
\begin{equation}\label{eq:rec2}
\pr(Y_{n+1} - Y_{n} = w, \rho_{n}<\infty |\mbox{ } \rho_{n-1} <\infty, Y_{n}=z,\mathcal{F}_{T_{Y_{n}}}) = p_{R}(w) ,
\end{equation}
for all $w,z\geq1$.

Clearly $\{N=n\} = \{\rho_{k}<\infty \mbox{ for all } k =1,\dots,n-1 \mbox{ and } \rho_{n} = \infty\}$, $n\geq2$, 
and hence, 
$$
\displaylines{
\left\{\textstyle{\bigcap \limits_{n=1}^{m}} \{Y_{n+1}-Y_{n} = w_{n}\}, N=m+1\right\} = \hfill \cr
=  \left\{\textstyle{ \bigcap\limits_{n=1}^{m}} \{Y_{n+1}-Y_{n} = w_{n}, \rho_{n}<\infty\}, \rho_{m+1}=\infty\right\},}
$$
 for all $m\geq1$, using this, from $(\ref{eq:rec1})$, $(m-1)$ applications of $(\ref{eq:rec2})$, and $(\ref{eq:Y2})$, since $p_{R}(w)= q\bar{p}_{R}(w)$, we have that
\begin{equation*}\label{eq:prod-ind}
\pr\left(\bigcap \limits_{n=1}^{m} \{Y_{n+1}-Y_{n} = w_{n}\}, N=m+1\right)= p q^{m} \prod \limits_{n=1}^{m} \bar{p}_{R}(w_{n}),  
\end{equation*}
for all $m\geq1$ and $w_{n}\geq1$.  From the last display and $(\ref{eq:YN})$, due to Proposition \ref{expbounds3scp}, we have that $Y_{N}$ is exponentially bounded by an elementary conditioning argument as follows. Letting $(\tilde{\rho}_{k}, \tilde{R}_{k}), k\geq1$ be independent pairs of random variables each of which is distributed as $(\rho,R)$ and the geometric random variable $\tilde{N} := \inf\{n\geq1: \tilde{\rho_{n}} = \infty\}$, we have that $Y_{N}$ is equal in distribution to $\sum \limits_{k=0}^{\tilde{N}-1} \tilde{R_{k}}$, $\tilde{R}_{0}:=1$.

%, on the same probability space equipped with a probability measure $\tilde{\pr}$,
%[We give an explicit proof of how this can be done. Consider the family of independent pairs of $\mbox{r.v.'s}$ $\{ (\tilde{\rho}_{k}, \tilde{R}_{k}), k\geq1\}$ each of which is distributed as $(\rho,R)$ on the same probability space equipped with a probability measure $\tilde{\pr}$, as explained in $\mbox{Chpt.}$ 8 of Williams (1991). 
%Define also $\tilde{N} = \inf\{n\geq1: \tilde{\rho_{n}} = \infty\}$.
%Using $(\ref{eq:prod-ind})$ we have that for all $n\geq2$,  
%$$
%\displaylines{
%\pr \left( \sum \limits_{k=2}^{n} (Y_{k}-Y_{k-1}) =w-1, N=n \right) = \hfill \cr  
%= pq^{n-1} \tilde{\pr} \left( \sum \limits_{k=1}^{n-1} (\tilde{R}_{k} + 1) =w-1 \vline \mbox{ }\bigcap\limits_{k=1}^{n-1} \{\tilde{\rho}_{k}<\infty\}\right), }
%$$
%where the latter is true for $n=2$ by elementary conditioning, while for general $n\geq2$ comes by iteration.
%Thus from $(\ref{eq:YN})$ and because $\pr(N=1) = p$, defining $R_{0}:=1$, the last display gives that 
%\begin{eqnarray*}
%\pr(Y_{N} = w) &=& pI_{\{w=1\}} + \sum \limits_{n\geq 2} \pr\left( \sum \limits_{k=2}^{n} (Y_{k}-Y_{k-1}) =w-1, N=n\right) \\
%& = & \tilde{\pr}\left( \sum \limits_{k=0}^{\tilde{N}-1} \tilde{R_{k}} = w\right), 
%\end{eqnarray*}
%however because by Lemma $\ref{lem4pip}$ we have that $\tilde{R}_{k}, k=0,\dots,\tilde{N}-1$ are exponentially bounded, 
%by further conditioning on $\tilde{N}$, it is easy to see that the proof is complete.]

We proceed to prove that $T_{Y_{N}}$ and $\inf_{t \leq T_{Y_{N}}}\bar{r}_{t}$ are exponentially bounded random variables. By $(\ref{Yinf})$, letting $\bar{x}_{t} = \sup_{s \leq t}\bar{r}_{s}$, we have that $\{T_{Y_{N}} > t\}= \{\bar{x}_{t} \leq Y_{N}\}$; from this and set theory we have that, for any $a>0$
\begin{eqnarray}\label{bpexm}
\pr(T_{Y_{N}} > t) &=& \pr(\bar{x}_{t} \leq Y_{N}) \nonumber\\
&\leq& \pr(\bar{x}_{t} < at) + \pr(\bar{x}_{t}\geq at, \bar{x}_{t} \leq Y_{N}) \nonumber\\
&\leq& \pr(\bar{x}_{t} < at) + \pr(Y_{N} \geq \lfloor at \rfloor),
\end{eqnarray}
for all $t\geq0$, where $\lfloor \cdot\rfloor$ is the floor function; choosing $a>0$ as in Proposition $\ref{shadldev2}$, because $\bar{x}_{t} \geq \bar{r}_{t}$, and since $Y_{N}$ is exponentially bounded, we deduce by (\ref{bpexm}) that  $T_{Y_{N}}$ is exponentially bounded as well.

Finally, we prove that $M:= \inf_{t \leq T_{Y_{N}}}\bar{r}_{t}$ is exponentially bounded. 
From set theory, 
\[
\pr(M < -x) \leq \pr \left(T_{Y_{N}} \geq \frac{x}{\mu} \right) + \pr\left(T_{Y_{N}} < \frac{x}{\mu}, \{ \bar{r}_{s} \leq -x \mbox{ for some } s \leq T_{Y_{N}}\}\right), 
\]
because $T_{Y_{N}}$ is exponentially bounded, it is sufficient to prove that the second term of the right hand side decays exponentially. However, recall that $\bar{r}_{_{T_{Y_{N}}}} \geq 1$, hence, 
$$\displaylines{ \pr\left(T_{Y_{N}} < \frac{x}{\mu}, \{ \bar{r}_{s} \leq -x \mbox{ for some } s \leq T_{Y_{N}}\}\right) \leq \hfill \cr 
\leq \pr\left((\bar{r}_{t}- \bar{r}_{s}) > x \mbox{ for some } s  \leq \frac{x} {\mu} \mbox{ and } t  \leq \frac{x} {\mu} \right),}
$$
where the term on the right of the last display decays exponentially in $x$, because $(\bar{r}_{t} - \bar{r}_{s})$, $t >s$ is bounded above in distribution by $\Lambda_{\mu}(s,t]$, the number of events of a Poisson process at rate $\mu$ within the time interval $(s,t]$, by use of standard large deviations for Poisson processes, because $\Lambda_{\mu}(s,t] \leq \Lambda_{\mu}( 0,x / \mu]$ for any $s,t \in (0,x / \mu]$. 
\end{proof}

The next lemma is used in the proof of Proposition $\ref{PROPiid}$ following.
%the following equality of events, 
\begin{lem}\label{Xbdownii}
Consider the setting of the definition of break points, Definition \ref{definbpts}. For all $n\geq1$, we have that
\begin{equation}\label{indXn}
\big\{ \textstyle{ \bigcap \limits_{l=1}^{n} } \{(X_{l}, \Psi_{l}, M_{l-1}) = (x_{l}, t_{l}, m_{l-1}) \}, \zeta_{t}^{O} \textup{ survives}\big\} = \{ \zeta_{t}^{[\eta_{z_{n}},w_{n}]}\textup{ survives}, \tau_{z_{n}} = w_{n}, A \},
\end{equation}
for some event $A \in \mathcal{F}_{w_{n}}$, where $z_{n} = \sum\limits_{l=1}^{n} x_{l}$ and $w_{n} = \sum\limits_{l=1}^{n} t_{l}$.
\end{lem}

\begin{proof}
Considering the setting of Lemma $\ref{bptsinf}$ we trivially have that 
\[
\{(Y_{N},T_{Y_{N}}, \inf_{t \leq T_{Y_{N}}}\bar{r}_{t}) = (x_{1}, t_{1}, m_{0})\} = \{\zeta_{t}^{[\eta_{x_{1}}, t_{1}]} \mbox{ survives}, T_{x_{1}} = t_{1}, B\},
\] 
for some event $B \in \mathcal{F}_{t_{1}}$; from this and Lemma $\ref{KeqK'}$ we have that 
$$\displaylines{  \{ (X_{1}, \Psi_{1}, M_{0}) = (x_{1}, t_{1}, m_{0}) ,\zeta_{t}^{O} \mbox{survives}\} \hfill \cr  
\hspace{7mm}=\{\zeta_{t}^{[\eta_{x_{1}},t_{1}]} \mbox{ survives},\tau_{x_{1}} = t_{1}, B, \zeta_{t}^{O} \mbox{survives}\} \cr 
= \{\zeta_{t}^{[\eta_{x_{1}},t_{1}]} \mbox{ survives},\tau_{x_{1}} = t_{1}, B, I_{t_{1}} \not=\emptyset\}
}
$$
for all $x_{1}\geq 1$, $t_{1} \in \R_{+}$, $m_{0} \geq0$, because $\{ I_{t_{1}} \not=\emptyset \} \in \mathcal{F}_{t_{1}}$ we have thus proved $(\ref{indXn})$ for $n=1$, for general $n\geq1$  the proof is derived by repeated applications of the last display.
\end{proof}

\begin{proof}[proof of Proposition $\ref{PROPiid}$]
Consider the setting of the definition of break points, Definition \ref{definbpts}.
Assume that $K_{n}$, $\tau_{K_{n}}$, $M_{n-1}$ are almost surely finite, we will prove that
\begin{eqnarray}\label{Xn}
\pr\left((X_{n+1},\Psi_{n+1},M_{n}) = (x, t,m)\vline\mbox{ } \textstyle{ \bigcap \limits_{l=1}^{n}} \{(X_{l},\Psi_{l},M_{l-1}) =(x_{l}, t_{l},m_{l-1})\}, \zeta_{t}^{O} \mbox{survives}\right)&& \nonumber\\
=\mbox{ } \pr\big( (X_{1},\Psi_{1},M_{0})= (x, t,m)|\mbox{ } \zeta_{t}^{O} \mbox{survives} \big)\hspace{10mm}
\end{eqnarray}
for all $(x_{l},t_{l},m_{l-1})$, $x_{l} \geq 1, t_{l}\in \R_{+}$, $m_{l-1}\geq0$, $l=1,\dots,n$, and hence in particular that $K_{n+1}$, $\tau_{K_{n+1}}$, $M_{n}$ are exponentially bounded.  By induction because $K_{1}$ and $\tau_{K_{1}},M_{0}$ are exponentially bounded by Proposition $\ref{PROPexpbnd}$ we have that $(\ref{Xn})$ completes the proof by Bayes's sequential formula. 

It remains to prove $(\ref{Xn})$, rewrite the conditioning event in its left hand side according to $(\ref{indXn})$ in Lemma $\ref{Xbdownii}$ and note that 
\[
\{\tau_{z_{n}} = w_{n}\} \subset \{ \zeta_{w_{n}}^{O}(z_{n})=1 \mbox{ and } \zeta_{w_{n}}^{O}(y)= -1, \mbox{ for all } y \geq z_{n}+1\},
\]
thus, applying Lemma $\ref{KeqK'}$, gives the proof by independence of the Poisson processes in disjoint parts of the graphical construction, because $(\zeta_{t+w_{n}}^{[\eta_{z_{n}}, w_{n}]}-z_{n})_{t\geq0}$ is equal in distribution to $(\zeta_{t}^{O})_{t\geq0}$ by translation invariance.
\end{proof}

\subsection{Proof of Theorem $\ref{THEthm}$}\label{S32}
We denote by $\bar{\pr}$ the probability measure induced by the construction of the process conditional on $\{\zeta_{t}^{O} \mbox{ survives}\}$ and, by $\bar{\E}$ the expectation associated to $\bar{\pr}$. Consider the setting of Theorem $\ref{THEprop}$ and let $\displaystyle{ \alpha =  \frac { \bar{\E} (r_{\tilde{\tau}_{1}})} { \bar{\E} (\tilde{\tau}_{1})}}$, $\alpha \in (0, \infty)$. %we refer to $\alpha$ as the \textit{limit of the speed}. 

\begin{proof}[proof of (i)]
Because $r_{\tilde{\tau}_{n}} = \sum\limits_{m = 1}^{n} (r_{\tilde{\tau}_{m}} - r_{\tilde{\tau}_{m-1}})$ and $\tilde{\tau}_{n} = \sum\limits_{m=1}^{n}( \tilde{\tau}_{m} - \tilde{\tau}_{m-1})$, $n\geq1$, using the strong law of large numbers twice gives us that
\begin{equation}\label{spdatbpts} 
\bar{\pr} \left( \lim_{n \rightarrow \infty} \frac{r_{\tilde{\tau}_{n}}}{\tilde{\tau}_{n}} = \alpha \right)=1,
\end{equation} 
we prove that indeed $\displaystyle{\lim_{t\rightarrow \infty}\frac{r_{t}}{t} =\alpha}$, $\bar{\pr}$ a.s..
From Theorem $\ref{THEprop}$ we have that
\begin{equation}\label{intt}
\frac{r_{\tilde{\tau}_{n}}-M_{n}}{\tilde{\tau}_{n+1}} \leq \frac{r_{t}}{t} \leq \frac{r_{\tilde{\tau}_{n+1}}} {\tilde{\tau}_{n}} , \mbox{ for all }  t\in[\tilde{\tau}_{n}, \tilde{\tau}_{n+1}), 
\end{equation}
$n\geq0$. Further, because $(M_{n})_{n\geq0}$, $M_{0}\geq0$, is a sequence of i.i.d.\ and exponentially bounded random variables we have that
\begin{equation}\label{spdM}
\bar{\pr}  \left(\lim_{n \rightarrow \infty} \frac{M_{n}}{n} =  0 \right) =1,
\end{equation}
by the 1st Borel-Cantelli lemma. Consider any $ a < \alpha$, by $(\ref{intt})$ we have that
\begin{equation}\label{eq:spdlessa}
\left\{ \frac{r_{t_{k}}}{t_{k}} < a  \mbox{ for some } t_{k} \uparrow \infty \right\}\subseteq \left\{ \limsup_{n\rightarrow \infty} \left\{\frac{r_{\tilde{\tau}_{n}}-M_{n}}{\tilde{\tau}_{n+1}} <  a \right\}\right\}, 
\end{equation}
however $\displaystyle{ \bar{\pr}\left(\limsup_{n\rightarrow \infty} \left\{\frac{r_{\tilde{\tau}_{n}}-M_{n}}{\tilde{\tau}_{n+1}} <  a \right\}\right)=0}$, to see this simply use $(\ref{spdM})$ and $(\ref{spdatbpts})$ to deduce that $\displaystyle{\lim_{n\rightarrow \infty}\frac{r_{\tilde{\tau}_{n}}-M_{n}}{\tilde{\tau}_{n+1}} = \alpha}$, $\bar{\pr}$ $\mbox{a.s.}$. By use of the upper bound in $(\ref{intt})$ and $(\ref{spdatbpts})$, we also have that for any $a > \alpha$, $\displaystyle{\bar{\pr}\left( \left\{ \frac{r_{t_{k}}}{t_{k}} > a  \mbox{ for some } t_{k} \uparrow \infty \right\}\right) =0}$, this completes the proof of \textit{(i)}. 
\end{proof}

\begin{proof}[proof of (ii)] We will prove that
\begin{equation*}
\lim_{t\rightarrow \infty}\bar{\pr} \left( \frac{r_{t} - \alpha t}{\sqrt{t}} \leq x \right) = \Phi\left(\frac{x}{\sigma^{2}}\right),
\end{equation*}
for some $\sigma^{2}>0$, $x \in\R$,  where $\Phi$ is the standard normal distribution function, $\mbox{i.e.}$,  $\displaystyle{ \Phi(y) := \frac{1}{\sqrt{2\pi}} \int_{-\infty}^{y}\exp\left(-\frac{1}{2} z^{2}\right)dz}$, $y\in \R$.

Define $N_{t}= \sup\{n: \tilde{\tau}_{n} <t\}$; evoking Lemma 2 in Kuczek \cite{K}, $\mbox{p.}$ 1330--1331, which applies due to Theorem $\ref{THEprop}$, we have that 
\begin{equation*}\label{kucinitNt}
\lim_{t\rightarrow \infty}\bar{\pr} \left( \frac{r_{N_{t}} - \alpha t}{\sqrt{t}} \leq x \right) = \Phi\left(\frac{x}{\sigma^{2}}\right),
\end{equation*}
$x \in\R$. From this, by standard association of convergence concepts, i.e.\ Slutsky's theorem, it is sufficient to show that
\begin{equation}\label{clt0}
\bar{\pr} \left(\lim_{t \rightarrow \infty } \frac{r_{t} - r_{N_{t}}}{\sqrt{t}}=0 \right)=1,
\end{equation}
and that $\sigma^{2}$ is strictly positive. Note however that, by Theorem $\ref{THEprop}$ we have that,
\begin{equation}\label{cltbouds}
\frac{ M_{\tilde{N}_{t}}} {\sqrt{t}} \leq \frac{ r_{t} - r_{N_{t}}}{\sqrt{t}} \leq \frac{ r_{\tilde{\tau}_{N_{t}+1}} - r_{\tilde{\tau}_{N_{t}}}}{\sqrt{t}} 
\end{equation}
for all $t\geq0$. 

We show that $(\ref{clt0})$ follows from $(\ref{cltbouds})$. Because $(r_{\tilde{\tau}_{n+1}} - r_{\tilde{\tau}_{n}})_{n\geq0}$, $r_{\tilde{\tau}_{1}}\geq1$, are $\mbox{i.i.d.}$ and exponentially bounded, by the 1st Borel-Cantelli lemma, and then the strong law of large numbers, we have that 
\[
\lim_{n\rightarrow \infty} \frac{ \frac{1}{\sqrt{n}}(r_{\tilde{\tau}_{n+1}} - r_{\tilde{\tau}_{n}})}{\sqrt{ \frac{\tilde{\tau}_{n}}{n}}}= 0
\] 
$\bar{\pr}$ $\mbox{a.s.}$, from the last display and emulating the argument given in $(\ref{eq:spdlessa})$ we have that 
$\displaystyle{ \lim_{t\rightarrow \infty} \frac{ r_{\tilde{\tau}_{N_{t}+1}} - r_{\tilde{\tau}_{N_{t}}}}{\sqrt{t}} =0}$, $\bar{\pr}$ $\mbox{a.s.}$.  Similarly, because $(M_{n})_{n\geq0}$, and $M_{0} \geq 0 $, are also $\mbox{i.i.d.}$ and exponentially bounded, we also have that 
$\displaystyle{\lim_{t\rightarrow \infty}\frac{ M_{\tilde{N}_{t}}}{\sqrt{t}} =0}$, $\bar{\pr} \mbox{ a.s.}$. 
%the final paragraph of 

Finally, we show that $\sigma^{2}>0$. As in the proof of Corollary 1 in Kuczek \cite{K}, because $\displaystyle{ \alpha = \frac { \bar{\E} (r_{\tilde{\tau}_{1}})} { \bar{\E} (\tilde{\tau}_{1})}}$, we need to show that 
$\bar{\E}\left( r_{\tilde{\tau}_{1}}\bar{\E}(\tilde{\tau}_{1})  - \tilde{\tau}_{1} \bar{\E} (r_{\tilde{\tau}_{1}})\right)^{2}>0$. However, because  $r_{\tilde{\tau}_{1}} \geq 1$, this follows by Chebyshev's inequality. This completes the proof of \textit{(ii)}. 

\end{proof}

For the remainder of the proof consider the graphical construction for $(\lambda, \mu)$ such that $\mu >\mu_{c}$ and $\mu\geq\lambda>0$. Consider $\zeta_{t}^{O}$, let $r_{t} =\sup I_{t}$ and $l_{t} = \inf I_{t}$ be  respectively the rightmost and leftmost infected of $I_{t}= \mathcal{I}(\zeta_{t}^{O})$. Consider also $\xi_{t}^{\Z}$, the contact process with parameter $\mu$ started from $\Z$. By Lemma $\ref{cccoup}$ we have that, for all $t\geq0$, 
\begin{equation}\label{coupHtZ}
I_{t} =  \xi_{t}^{\Z} \cap [l_{t},r_{t}] \mbox{ \textup{on} } \{I_{t} \not= \emptyset\}.
\end{equation}

\begin{proof}[proof of (iii)]
Let $\theta= \theta(\mu) >0$ be the density of the upper invariant measure, i.e., $\displaystyle{ \theta= \lim_{t\rightarrow \infty}\pr( x \in \xi_{t}^{\Z})}$. 
We prove that $\displaystyle{\lim_{t\rightarrow \infty}\frac{|I_{t}|}{t} = 2 \alpha \theta,}$ $\bar{\pr}$ a.s.. 

Considering the interval $[\max\{l_{t}, -\alpha t\}, \min\{r_{t},\alpha t\}]$, we have that for all $t\geq0$,
\begin{equation}\label{sllnineq}
\vline \mbox{ } \sum_{x=l_{t}}^{r_{t}} 1_{\{ x \in \xi_{t}^{\Z} \}} - \sum_{x =- \alpha t}^{ \alpha t} 1_{\{ x \in \xi_{t}^{\Z} \}} \mbox{ }\vline \leq | r_{t} - \alpha t| + | l_{t} + \alpha t|, \mbox{ on }\{I_{t} \not= \emptyset \},
\end{equation}
where we denote by $1_{E}$ the indicator of event $E$. However, by $(\ref{coupHtZ})$, we have that $\displaystyle{ |I_{t}|= \sum_{x=l_{t}}^{r_{t}} 1_{\{ x \in \xi_{t}^{\Z} \}}}$,  $\mbox{on } \{I_{t} \not= \emptyset \}$,  thus, because $\displaystyle{\lim_{t\rightarrow \infty} \frac{r_{t}}{t} = \alpha}$ and, by symmetry, $\displaystyle{  \lim_{t\rightarrow \infty} \frac{l_{t}}{t} =-\alpha}$, $\bar{\pr}$ a.s., the proof follows from $(\ref{sllnineq})$ because it is known that, for any $a>0$, $\displaystyle{ \lim_{t\rightarrow \infty} \frac{1}{t}\sum_{|x| \leq a t} 1_{\{x \in \xi_{t}^{\Z} \}} = 2 a \theta}$, $\pr$ a.s. (see equation (19) in the proof of Theorem 9 of Durrett and Griffeath \cite{DG}).
\end{proof}

\begin{proof}[proof of (iv)] Let $\rho = \inf\{t\geq0: I_{t} = \emptyset\}$. In the context of set valued processes, by general considerations, see Durrett \cite{D95}, it is known that weak convergence is equivalent to convergence of finite dimensional distributions and that, by inclusion-exclusion,  it is equivalent to show that for any finite set of sites $F\subset \Z$ 
\begin{equation*}\label{eq:compconv}
\lim_{t\rightarrow \infty} \pr(I_{t} \cap F = \emptyset) = \pr(\rho < \infty) + \pr(\rho = \infty)\phi_{F}(\emptyset),
\end{equation*}
where $\displaystyle{ \phi_{F}(\emptyset):= \lim_{t\rightarrow \infty} \pr(\xi_{t}^{\Z} \cap F =\emptyset)}$. By set theory we have that it is sufficient to prove 
$\displaystyle{ \lim_{t\rightarrow \infty} \pr(I_{t} \cap F = \emptyset, \rho \geq t)= \pr(\rho = \infty)\phi_{F}(\emptyset)}$,  because $\{\rho <t\} \subseteq \{I_{t}\cap F =\emptyset\}$. However, emulating the proof of the respective result for the contact process (see $\mbox{e.g.}$ Theorem 5.1 in Griffeath \cite{G}), we get $\displaystyle{\lim_{t \rightarrow \infty} \pr(\xi_{t}^{\Z} \cap F = \emptyset, \rho \geq t) = \pr(\rho =\infty)\phi_{F}(\emptyset)}$, hence, it is sufficient to prove that 
\begin{equation}\label{cccoupconseq}
\limsup_{t \rightarrow \infty} \pr(I_{t} \cap F = \emptyset, \rho \geq t) \leq  \lim_{t\rightarrow \infty}\pr(\xi_{t}^{\Z} \cap F = \emptyset, \rho \geq t),
\end{equation}
because also $\{ I_{t} \cap F = \emptyset, \rho \geq t\} \supseteq \{ \xi_{t}^{\Z} \cap F = \emptyset, \rho \geq t\}$,  by $(\ref{coupHtZ})$. 

It remains to prove $(\ref{cccoupconseq})$. By elementary calculations, 
\begin{equation*}\label{ccrhoinf}
\pr(I_{t} \cap F = \emptyset, \rho = \infty) - \pr(\xi_{t}^{\Z} \cap F = \emptyset, \rho \geq t) \leq \pr(  \xi_{t}^{\Z} \cap F \supsetneq I_{t} \cap F, \rho = \infty), 
\end{equation*}
for all $t\geq0$, where we used that by $(\ref{coupHtZ})$, $I_{t} \subset \xi_{t}^{\Z}$ for all $t\geq0$. From the last display above and set theory we have that 
$$\displaylines{\pr(I_{t} \cap F = \emptyset, \rho \geq t)-  \pr(\xi_{t}^{\Z} \cap F = \emptyset, \rho \geq t) \hfill \cr
\hfill \leq \pr( \xi_{t}^{\Z} \cap F \supsetneq I_{t} \cap F, \rho = \infty) + \pr( t<\rho <\infty),  }$$
for all $t\geq0$, however the limit as $t\rightarrow \infty$ of both terms of the right hand side in the above display is $0$, for the former this comes by $(\ref{coupHtZ})$, because $\displaystyle{\lim_{t\rightarrow \infty} r_{t} = \infty}$ and $\displaystyle{\lim_{t\rightarrow \infty} l_{t} = \infty}$, $\bar{\pr}$ $\mbox{a.s.}$, while for the latter this is obvious.

\end{proof}


\begin{thebibliography}{99}

\bibitem{BG} { \sc Bezuidenhout, C.E.} and { \sc Grimmett, G.R.} (1990). The critical contact process dies out. \textit{Ann. Probab.} \textbf{18} 1462--1482.

\bibitem{D80} {\sc  Durrett, R.} (1980). On the growth of one-dimensional contact processes. {\em Ann. Probab.} {\bf 8} 890--907.

\bibitem{D84}  {\sc Durrett, R.}  (1984). Oriented percolation in two dimensions. {\em Ann. Probab.} {\bf 12} 999--1040.

\bibitem{D88}  {\sc Durrett, R.}(1988). {\em Lecture Notes on Particle Systems and Percolation.} Wadsworth.

\bibitem{D95}  {\sc  Durrett, R.} (1995). { \em Ten lectures on particle systems} Lecture Notes in Math. {\bf 1608}, Springer-Verlag, New York.
\bibitem{DG}  {\sc  Durrett, R.} and  { \sc Griffeath, D.} (1983). Supercritical contact processes on $\Z$.  {\em Ann. Probab.} {\bf 11} 1--15.
\bibitem{DS}   {\sc  Durrett, R.}  and { \sc Schinazi, R. }(2000). Boundary modified contact processes. {\em J. Theoret. Probab.} {\bf 13} 575-594.
\bibitem{G}  { \sc Griffeath, D.} (1979). {\em Additive and cancelative interacting particle systems. } Lecture Notes in Math. {\bf 724} Springer-Verlag, Berlin.

\bibitem{H}  {\sc Harris, T.E.} (1972). Nearest neighbor Markov interaction processes on multidimensional lattices. {\em Adv. in Math.} {\bf 9} 66--89.

\bibitem{K} {\sc Kuczek, T.} (1989). The central limit theorem for the right edge of supercritical oriented percolation. {\em Ann. Prob.} {\bf 17} 1322--1332.

\bibitem{L}  {\sc Liggett, T.} (1985). {\em Interacting particle systems.} Springer, New York.
\bibitem{L99}  {\sc  Liggett, T.} (1999). {\em Stochastic Interacting Systems: Contact, Voter and Exclusion Processes.} Springer, New York.
\bibitem{S}  {\sc Stacey, A.} (2003). Partial immunization processes. {\em Ann. Appl. Probab.} {\bf 13}, 669--690.


\end{thebibliography}
\end{document}